\documentclass{amsart}
\pdfoutput=1
\usepackage{amsmath,amsfonts,amsthm,graphicx}
\usepackage{hyperref}
\usepackage{comment}
\usepackage{graphicx,psfrag}  
\usepackage[psamsfonts]{amssymb}
\usepackage{amscd}
\usepackage{color}
\usepackage[all,cmtip]{xy}
\usepackage{mathtools}

\usepackage{tikz}

\title[Periodic Data Rigidity]{Dominated splitting from constant periodic data and Global rigidity of Anosov automorphisms}
\author{Jonathan DeWitt}
\address[DeWitt]{Department of Mathematics, The University of Maryland, College Park, MD, USA, 20742}
\email{dewitt@umd.edu}
\author{Andrey Gogolev}
\address[Gogolev]{Department of Mathematics, The Ohio State University, Columbus, OH, USA, 43210}
\email{gogolyev.1@osu.edu}

\date{\today}
\theoremstyle{plain}

\newtheorem{theorem}{Theorem}[section]
\newtheorem{thmy}{Theorem}

\newtheorem{proposition}[theorem]{Proposition}
\newtheorem{lemma}[theorem]{Lemma}
\newtheorem{corollary}[theorem]{Corollary}

\newtheorem{quest}[theorem]{Question}
\newtheorem{claim}[theorem]{Claim}
\newtheorem{definition}[theorem]{Definition}
\theoremstyle{remark}
\newtheorem{remark}[theorem]{Remark}

\newtheorem{observation}[theorem]{Observation}

\def\eps{\varepsilon}

\def\Diff{\operatorname{Diff} }

\def\title{\em}

\def\cU{\mathcal{U}}

\def\cA{\mathcal{A}}

\renewcommand{\epsilon}{\varepsilon}

\newcommand{\SL}{\operatorname{SL}}
\newcommand{\GL}{\operatorname{GL}}

\newcommand{\abs}[1]{\left| #1 \right|}
\newcommand{\mc}[1]{\mathcal{ #1 }}

\newcommand{\wt}[1]{\widetilde{ #1 }}

\def\transverse{\,\raise2pt\hbox to1em{\hfil$\top$\hfil}\hskip -1em \hbox
to1em{\hfil$\cap$\hfil}\,} 
\newcommand\Z{\mathbb{Z}}
\newcommand\R{\mathbb R}

\newcommand\TT{{\mathbb T}}
\newcommand\N{{\mathbb N}}

\newlength{\figboxwidth} 

\begin{document}

\begin{abstract} We show that a $\GL(d,\R)$ cocycle over a hyperbolic system with constant periodic data has a dominated splitting whenever the periodic data indicates it should.
This implies global periodic data rigidity of generic Anosov automorphisms of $\mathbb{T}^d$. Further, our approach also works when the periodic data is narrow, that is, sufficiently close to constant. We can show global periodic data rigidity for certain non-linear Anosov diffeomorphisms in a neighborhood of an irreducible Anosov automorphism with simple spectrum.
\end{abstract}

\maketitle

\section{Introduction}

Anosov diffeomorphisms are a well-studied class of examples in the field of dynamical systems. By definition an Anosov diffeomorphism is a diffeomorphism of a Riemannian manifold $M$ such that $TM$ has a continuous splitting into a complementary pair of $Df$-invariant bundles $E^s$ and $E^u$, the vectors of which are uniformly contracted or expanded, respectively, by $Df$. We refer to $E^s$ and $E^u$ as the stable and unstable bundles of $f$, respectively. There is a simple algebraic construction of Anosov diffeomorphisms: take the action of a matrix $L\in \GL(d,\Z)$ that has no eigenvalues of modulus $1$ on the torus $\R^d/\Z^d$. Diffeomorphisms thus obtained are called \emph{hyperbolic toral automorphisms} or \emph{Anosov automorphisms}. 

Anosov automorphisms have strong rigidity properties. One of the most basic, due to Franks and Manning, is that if $f$ is an Anosov diffeomorphism that is in the same free homotopy class as an Anosov automorphism $L$, then there is a homeomorphism $h$ such that $hfh^{-1}=L$, see e.g.~\cite[Sec.~2.3]{katok1995introduction}. We call $h$ a \emph{conjugacy}.  
If $h$ were a $C^1$ diffeomorphism, then it would imply that for every $f$-periodic point $p$ of period $n$ the derivative $D_pf^n$ is conjugate as a linear map to $L^n$ with the conjugacy given by $D_ph$. 
Consequently, if there is a periodic point where $D_pf^n$ is not conjugate to $L^n$, then any conjugacy $h$ cannot be a $C^1$ diffeomorphism. Hence we may view each periodic point $p$ as coming with a possible obstruction to the differentiability of $h$. If all of these obstructions vanish, then we say that $f$ has the same \emph{periodic data} as $L$. We say that $L$ is $C^{\infty}$ \emph{periodic data rigid} if for any $C^{\infty}$ Anosov diffeomorphism $f$ in the same homotopy class as $L$, if $f$ has the same periodic data as $L$, then $f$ is $C^{\infty}$ conjugate to $L$.  One of our main results is the following theorem.

\begin{theorem}\label{thm:rigidity_dim_3}
Suppose that $L$ is an Anosov automorphism of $\mathbb{T}^3$. Then $L$ is $C^{\infty}$ periodic data rigid.
\end{theorem}

\noindent In the case where an  automorphism $L$ is conformal on its stable and unstable bundles, this result is due to Kalinin and Sadovskaya \cite{kalinin2009anosov}.

We remark that Theorem \ref{thm:rigidity_dim_3} is not perturbative: we do not consider only those $f$ that are $C^1$ close to $L$. That said, this result builds upon and relies heavily on the large amount of work  for the local, perturbative version of the problem, which asks whether any perturbation of an Anosov automorphism that has the same periodic data is $C^1$ conjugate back to the original automorphism. The main obstacle to proving global versions of existing local results is the construction of a particular type of $Df$-invariant splitting of the unstable bundle.

The main contribution of this article to the periodic data rigidity program is overcoming the problems that arise when not working perturbatively. We explain what the issue is in the simplest setting where it occurs, namely, in dimension 3. Let $L\colon\TT^3\to\TT^3$ be an Anosov automorphism with real eigenvalues of distinct absolute values. If one has an Anosov diffeomorphism on $\mathbb{T}^3$ with a $2$-dimensional unstable bundle $E^u$ homotopic to $L$, then, a priori, there is no reason to expect that $E^u$ has a splitting into two $Df$-invariant subbundles. However, note that the unstable bundle of $L$ does split into two subbundles corresponding to the two unstable eigenvalues. If $f$ is an Anosov diffeomorphism with the same periodic data as $L$, then the main remaining difficulty for establishing periodic data rigidity is showing that the unstable bundle of $f$ has a splitting with the same properties as the splitting of the unstable bundle of $L$.

In this paper, we show how to construct a $Df$-invariant splitting of $E^u$ from just the periodic data of $Df$.
The type of splitting of $E^u$ that we seek is called a \emph{dominated splitting}, which we now define. Suppose that $\sigma\colon X\to X$ is a homeomorphism of a compact metric space. If $A\colon X\to \GL(d,\R)$ is a continuous map, then we may consider the products of the matrices $A(x)$ along a trajectory of $\sigma$. We write $A^n(x)=A(\sigma^{n-1}(x))\cdots A(x)$. We say that a continuous splitting $E\oplus F$ of the bundle $X\times \R^d$ is $A$-invariant if $A(x)E(x)=E(\sigma(x))$ and $A(x)F(x)=F(\sigma(x))$. A continuous splitting $E\oplus F$ is called a \emph{dominated splitting} if this splitting is $A$-invariant and there exist constants $C>0$ and $0<\tau<1$ such that for every $x\in X$ and $n\in \N$, 
\[
\|A^n(x)\vert F_x\|<C\tau^nm(A^n(x)\vert E_x),
\]
where $m(A)=\|A^{-1}\|^{-1}$ denotes the conorm of the matrix $A$.  We say that the dominated splitting {\it has index $k$} if $\dim E=k$.
See \cite{bochi2009some} for a useful discussion of this and related notions. 

A closely related notion is that of uniform hyperbolicity. We say that a linear cocycle $A\colon X\to \GL(d,\R)$ over a homeomorphism $(X,\sigma)$ as above is \emph{uniformly hyperbolic} if there exists a non-trivial continuous $A$-invariant splitting $E^s\oplus E^u$, a continuous metric $\|\cdot\|_{x}$ on the vector bundle $X\times \R^d$, and $\lambda>1$ such that for each $x\in X$,
\begin{equation}
\|A(x)\vert_{E^s}\|< \lambda^{-1}<1<\lambda< m(A(x)\vert_{E^u}).
\end{equation}
In words, vectors in $E^s$ are uniformly contracted and those in $E^u$ are uniformly expanded. Formally, an \emph{Anosov diffeomorphism} $f\colon M\to M$ is a diffeomorphism of a Riemannian manifold for which the cocycle $Df\colon TM\to TM$ is uniformly hyperbolic. One defines cocycles and dominated splittings completely analogously if the bundle is not a trivial bundle.

As outlined above, the obstacle to making progress in the global rigidity problem for Anosov automorphisms is using the periodic data to construct a dominated splitting. In the local problem, such a splitting always exists for perturbative reasons, see, for example, \cite[Thm.~2.3]{dewitt2021local}. 

Our main technical result is the following, which produces such a splitting. We work in the more abstract setting of cocycles over subshifts of finite type because the arguments in this context are more transparent. 
We say that a cocycle $A\colon \Sigma\to \GL(d,\R)$ over a subshift of finite type 
 $\Sigma$ has \emph{constant} periodic data if there is a list of numbers $\lambda_1\ge \cdots\ge \lambda_d$ such that for each periodic point $p$ of period $n$ the moduli of the eigenvalues of $A^n(p)$ are $e^{n\lambda_1},\ldots,e^{n\lambda_d}$. For the formal definition of constant periodic data, see Definition \ref{defn:constant_periodic_data}. Our main theorem is the following.
\begin{theorem}\label{thm:dominated_splitting_constant_data}
Suppose that $\Sigma$ is a transitive, invertible, subshift of finite type and that $A\colon \Sigma\to \GL(d,\R)$ is a H\"older continuous cocycle with constant periodic data associated to exponents $\lambda_1\ge \cdots \ge \lambda_d$. If $\lambda_k>\lambda_{k+1}$, then $A$ has a dominated splitting of index $k$.
\end{theorem}

 In fact, we are able to prove an even stronger result, which allows us to produce dominated splittings even when the periodic data is not constant, but instead is ``narrow," in the sense that it lies sufficiently close to constant periodic data. We formally define $\delta$-narrow periodic data in Definition \ref{defn:narrow_periodic_data}, but we remark here that small perturbations of Anosov automorphisms have narrow periodic data.

 \begin{theorem}\label{thm:narrow_data}
    Suppose that $(\Sigma,\sigma)$ is a transitive, invertible, subshift of finite type and let $\lambda_1\ge \lambda_2\ge \cdots \ge \lambda_d$ with  $\lambda_k>\lambda_{k+1}$.  For each $\beta\in (0,1)$ there exists $\delta>0$ such that if $A$ is a $\beta$-H\"older $\GL(d,\R)$ cocycle with $\delta$-narrow periodic data centered at $(\lambda_1,\ldots,\lambda_d)$ then $A$ has a dominated splitting of index $k$.
\end{theorem}

As far as the authors are aware, the actual construction of a dominated splitting from periodic data has not been carried out except in systems satisfying particular bunching conditions. One of the only examples the authors are familiar with is due to Butler \cite{butler2017characterizing} in the context of the rigidity of the geodesic flow on complex hyperbolic space.  Butler's argument makes use of the 1:2 ratio between Lyapunov exponents in this setting and  involves delicate control of the regularity of the bundles he is considering \cite[Lem.~3.10]{butler2017characterizing}. In contrast, as we shall see, for constant periodic data our approach works for bundles that are merely H\"older with no control at all needed on the H\"older exponent. Other results constructing dominated splittings are known under other assumptions. For example, with fiber bunching, an additional constraint on the cocycle requiring $\|A\|/m(A)$ to be small, Velozo Ruiz constructed a dominated splitting for $\SL(2,\R)$ cocycles over subshifts of finite type under an assumption which is equivalent to the following: there exists $\epsilon>0$ such that for each periodic point $p^n$ of period $n$, $A^n(p^n)$ has an eigenvalue of modulus at least $e^{\epsilon n}$ \cite{velozo2020characterization}. However, as Velozo Ruiz demonstrates in that same paper, that without fiber bunching the same assumption on the eigenvalues of periodic points is insufficient to imply the existence of a dominated splitting. (For a more geometric example in the smooth setting see~\cite[Thm.~1]{gogolev2010diffeomorphisms}.) { In \cite{chen2023uniformity}, those authors obtain that for cocycles that are twisting and admit holonomies, e.g.~are fiber bunched, that there is uniform growth of the norm of the cocycle if one assumes that the largest exponent of every periodic point is the same.}
The construction of dominated splittings is also of interest in the field of Anosov representations. For example,
in order to study Anosov representations, Kassel and Potrie construct a dominated splitting for locally constant cocycles that have a uniform gap in their periodic data \cite{kassel2022eigenvalue}.

\subsection{Applications to non-constant periodic data.}

Perhaps the most interesting application of our approach is that we are able to prove global periodic data rigidity for systems that are $C^1$ close to Anosov automorphisms. When we talk about periodic data rigidity of Anosov diffeomorphisms that are not Anosov automorphisms, the problem is stated somewhat differently. In this case, if we have two Anosov diffeomorphisms on a torus in the same free homotopy class, then as before, there is a H\"older conjugacy $h$ between them. If $f$ is not an Anosov automorphism, then the periodic data may vary among periodic points of a given period. Thus the conjugacy $h$ provides a ``marking" of the points. Consequently, we say that $f$ and $g$ have the same \emph{periodic data with respect to a conjugacy $h$}, if for each periodic point $p$ of $f$ of period $k$ the linear maps $D_pf^k$ and $D_{h(p)}g^k$ are conjugate. We say that an Anosov diffeomorphism is $C^{1+{\text{H\"older}}}$ rigid if whenever another $C^2$ Anosov diffeomorphism is conjugate to it by a conjugacy $h$ and the diffeomorphism has the same periodic data with respect to the conjugacy, then $h$ is a $C^{1+\text{H\"older}}$ diffeomorphism.

We say that an Anosov automorphism is {\it irreducible} if the characteristic polynomial of $L\in \GL(d,\Z)$ is irreducible over $\Z$. This is a standard and often necessary assumption when studying rigidity in higher dimensions. 

 We deduce a global version of the main result of the second author's thesis~\cite[Thm.~A]{gogolev2008smooth}. 

\begin{theorem}\label{thm:non_linear_rigidity}
Let $L$ be an irreducible hyperbolic automorphism of $\TT^d$, $d\ge 3$, with simple Lyapunov spectrum. Then there exists a $C^1$-neighborhood $\cU\subseteq \Diff^{2}(\mathbb{T}^d)$ of L such that any $f \in\cU$ satisfying Property $\cA$ and any $C^2$ Anosov diffeomorphism $g$ with the same periodic data as $f$ are $C^{1+\textup{H\"older}}$ conjugate.
\end{theorem}

A perturbation of $L$ has the same intermediate-speed foliations that $L$ does and Property $\cA$, which conjecturally always holds, asserts that these foliations are transitive. Examples of open sets of diffeomorphisms which satisfy Property $\cA$ in dimensions up to 5 were given in~\cite[Sections 4.2 and 5]{gogolev2008smooth}. In fact, in dimension 3, Property $\mc{A}$ always holds. So, we can actually conclude a stronger result, generalizing~\cite{gogolev2008differentiable}.
\begin{theorem}\label{thm:nonlinear_rigidity_3}

Fix exponents $\lambda_1>\lambda_2>0>\lambda_3$. Then there exists $\delta>0$ such that if $f$ is any $C^2$ Anosov diffeomorphism of $\mathbb{T}^3$ with $\delta$-narrow spectrum centered at $\lambda_1,\lambda_2,\lambda_3$, then $f$ is globally $C^{1+\text{H\"older}}$ periodic data rigid. In particular, this implies that there is a $C^1$ open neighborhood of any Anosov automorphism of $\mathbb{T}^3$ with simple Lyapunov spectrum such that every member of that neighborhood is periodic data rigid.
\end{theorem}

The proof of the above theorems follows from applying our Theorem~\ref{thm:narrow_data} to the approach in Theorem~A of~\cite{gogolev2008differentiable, gogolev2008smooth}. Indeed, for Theorem \ref{thm:non_linear_rigidity}, for any $\delta>0$ by choosing a sufficiently small $\cU$ we have that $f$ has $\delta$-narrow periodic data. By the main assumption, the diffeomorphism $g$ also has $\delta$-narrow periodic data and, hence, Theorem~\ref{thm:narrow_data} yields a dominated splitting for $g$. After this the proof in~\cite{gogolev2008smooth} goes through even when $g\notin \cU$.

\subsection{Applications to smooth rigidity of Anosov automorphisms}

Theorem \ref{thm:rigidity_dim_3} is a special case of a more general, higher dimensional result which continues the work on higher dimensional toral automorphisms by the second author, Guysinsky, Kalinin and Sadovskaya~\cite{gogolev2008differentiable, gogolev2008smooth, kalinin2009anosov, gogolev2011local} and provides a global version of the higher dimensional result of~\cite{gogolev2011local}. 
\begin{proposition}\label{cor:generic_anosov_automorphism}
Let $L \colon \TT^d \to \TT^d$ be an irreducible Anosov automorphism such that no three of its eigenvalues have the same modulus. Then $L$ is $C^{1+\textup{H\"older}}$ rigid.
\end{proposition}

Anosov automorphisms are also known to exist on nilmanifolds, which are compact quotients of a nilmanifold by a lattice. 
Work of the first author~\cite{dewitt2021local} in this setting can be now globalized which we can briefly formulate as follows. The conditions of ``irreducibility" and ``sorted spectrum" are defined in~\cite{dewitt2021local}, but we note here that they are necessary for even local rigidity.

\begin{corollary} \label{cor:nilmanifolds}
Suppose that $A\colon N/\Gamma\to N/\Gamma$ is an irreducible Anosov automorphism of a nilmanifold $N/\Gamma$ with simple, sorted Lyapunov spectrum. Then $A$ is globally $C^{1+\textup{H\"older}}$ periodic data rigid.
In particular, this shows that there exist non-toral nilmanifold Anosov automorphisms in arbitrarily high dimension which are $C^{1+\textup{H\"older}}$ rigid.
\end{corollary}

\subsection{Relationship to prior work} Now we provide some commentary on how the current work fits into the rigidity program. Until very recently, the strongest local periodic data rigidity result available in higher dimensions was \cite{gogolev2011local}, which requires that the Anosov automorphism has no more than three eigenvalues with the same modulus, as well as be irreducible. This builds on earlier work such as \cite{gogolev2008differentiable, gogolev2008smooth}, which studies the case where each eigenvalue of the automorphism has a distinct modulus. Those results produced a conjugacy that is $C^{1+\text{H\"older}}$, though not necessarily one that is $C^{\infty}$. There is another line of work that shows that one may ``bootstrap" a $C^1$ conjugacy to a $C^{\infty}$ conjugacy. One of results in this direction is \cite{gogolev2017bootstrap}, which takes advantage of strong Diophantine properties of linear invariant foliations and bootstraps a $C^1$ conjugacy between an Anosov automorphism of $\mathbb{T}^3$ and its $C^1$ small perturbation all the way to $C^{\infty}$. Recently Kalinin, Sadovskaya, and Wang showed that it is possible to bootstrap a $C^1$ conjugacy between an Anosov automorphism and a $C^k$ small perturbation all the way to a $C^{\infty}$ conjugacy \cite{kalinin2023smooth}. Due to the KAM scheme used in that paper, $k$ may be very large, however.
Even more recently Zhenqi Wang announced at the Spring 2023 Maryland Dynamics conference that local rigidity for $C^k$ small perturbations holds for all toral automorphisms with irreducible characteristic polynomial in any dimension without any spectral restrictions \cite{wang_rigidity}. Similarly to~ \cite{gogolev2017bootstrap}, this result also strongly utilizes Diophantine properties of invariant foliations and, similarly to~\cite{kalinin2023smooth}, requires that the perturbation be small in a high regularity norm in order to run a KAM scheme. 

Some global rigidity results were already known due to work of Kalinin and Sadovskaya \cite{kalinin2009anosov}, which will be explained during the proof of Theorem \ref{thm:rigidity_dim_3}. 
These results apply to hyperbolic toral automorphisms in dimension $3$ that have a conjugate pair of complex eigenvalues, as well as hyperbolic toral automorphisms in dimension $4$ for which the eigenvalues of the automorphism have only two distinct moduli. In each case, these assumptions are essentially that the eigenvalues of the automorphism have exactly two distinct moduli: one contracting, one expanding. There are certain other earlier results due to de la Llave that imply global rigidity under the assumption that the entire stable and unstable bundles of a non-linear Anosov diffeomorphism be conformal \cite{delallave2002rigidity}. The only dimension in which unrestricted, non-perturbative $C^{\infty}$ periodic data rigidity of toral automorphisms was previously shown is dimension $2$, which goes back to the late 1980s and is due to work of de la Llave, Marco, and Moriy\'{o}n \cite{delallave1992smooth, marco1987invariants, delallave1987invariants}.

We also mention that recently the second author and F.~Rodriguez Hertz developed a new matching functions approach to rigidity of Anosov diffeomorphisms which disregards the dominated splitting structure, even when it is present, and yields global results~\cite{gogolev2023smooth}. This approach only applies to highly non-linear (very non-algebraic) Anosov diffeomorphisms with one dimensional stable foliation. Still there is some overlap between our Theorem~\ref{thm:nonlinear_rigidity_3} and~\cite[Theorem 1.5]{gogolev2023smooth}. Namely, under certain further spectral assumptions on $L\colon\TT^3\to\TT^3$ the method of~\cite{gogolev2023smooth}  gives $C^\infty$ periodic data rigidity for diffeomorphisms in an open and dense subset of a neighborhood of $L$, but not necessarily all diffeomorphisms in an open neighborhood of $L$ as is obtained in Theorem \ref{thm:nonlinear_rigidity_3}.

\subsection{Remarks on the proof}

We now say a couple words about the ideas behind the proof of Theorem \ref{thm:dominated_splitting_constant_data}. In order to construct a dominated splitting, we use a characterization of domination that we may verify for each orbit individually. We will verify this criterion by using a shadowing argument. The characterization of uniform hyperbolicity is due to Yoccoz \cite{yoccoz2004some} in dimension two and was extended by Bochi and Gourmelon \cite{bochi2009some} to a characterization of domination in higher dimensions.
\begin{theorem}\label{thm:yoccoz_criterion} 
\cite[Thm.~A]{bochi2009some}
Suppose that $f\colon X\to X$ is a homeomorphism of a compact metric space and that $\mc{E}$ is a continuous vector bundle over $X$. For a continuous cocycle $\mc{A}\colon \mc{E}\to \mc{E}$, the following are equivalent:
\begin{enumerate}
\item 
$\mc{A}$ has a dominated splitting of index $k$,
\item 
There exist $C>0$ and $\tau<1$ such that $\frac{\sigma_{k+1}(\mc{A}(n,x))}{\sigma_k(\mc{A}(n,x))}<C\tau^n$ for all $x\in X$ and $n\ge 0$, where $\mc{A}(x,n)$ is defined in equation \eqref{defn:cocycle_iterate} and $\sigma_k(\mc{A}(n,x))$ denotes the $k$th singular value of $\mc{A}(n,x)$.
\end{enumerate}

\end{theorem}
This is essentially the statement that the singular values of $\mc{A}(n,x)$ separate exponentially fast. 
The proof of Theorem \ref{thm:dominated_splitting_constant_data} will proceed by using a shadowing argument to check the criterion in Theorem \ref{thm:yoccoz_criterion}. We will shadow an arbitrary trajectory by a periodic trajectory and then compare the two trajectories for a small amount of time near the point where they are closest. Due to the tight constraints on the periodic data, this yields usable estimates on the growth of singular values along that arbitrary trajectory.

\

\noindent \textbf{Acknowledgements:} The authors are grateful to Daniel Mitsutani for many helpful discussions. The authors are also grateful to Dmitry Dolgopyat and Amie Wilkinson for helpful comments on the manuscript. The authors are grateful to the referee for many valuable comments and suggestions. The first author was supported by the National Science Foundation under Award No.~DMS-2202967. The second author was partially supported by NSF award DMS-2247747.

After the paper was finished, the authors learned from Jairo Bochi during a visit to Penn State that Misha Guysinsky announced a similar result some years ago for $\SL(2,\R)$ cocycles \cite[Thm.~1.2]{sadovskaya2013cohomology}. We are grateful to Misha Guysinsky for subsequent discussion, which revealed that he has some different techniques that yield similar results and may be used to address a related question of Velozo Ruiz \cite{velozo2020characterization} about cocycles that are close to fiber bunched \cite{guysinsky2023personal}.

\section{Preliminaries}

In this section we review some standard definitions and set notation that will be used through the rest of the paper.

\subsection{Subshifts of finite type}

We now give some definitions concerning subshifts of finite type. For a natural number $m$, we let $Q$ be an $m\times m$ matrix with entries $\{0,1\}$. We may consider the space of all bi-infinite strings of symbols $\{1,\ldots,m\}^{\Z}$. Then we restrict to the subspace only allowing the substrings admitted by $Q$. For $\omega\in \{1,\ldots,m\}^{\Z}$, we write its elements as $(\ldots, \omega_{-1},\omega_0,\omega_1,\ldots)$. A \emph{subshift of finite type} or SFT is the subspace:
\[
\Sigma =\{\omega\in \{1,\ldots, m\}^{\Z}: Q_{\omega_i \omega_{i+1}}=1\}.
\]
The dynamics on this subspace is the left shift $\sigma$, defined by
\[
(\sigma(\omega))_i=\omega_{i+1}.
\]
We require that for every pair $(i,j)$ there is some $\ell_{ij}$ with $Q^{l_{ij}}_{ij}>0$. This is equivalent to $(\Sigma,\sigma)$ being transitive.

There is also a natural metric on $\Sigma$. Write
\[
N(\omega,\eta)=\min \{\abs{n}:\omega_n\neq \eta_n\},
\]
then for $\omega\neq \eta$ we define the metric:
\[
d(\omega,\eta)=e^{-N(\omega,\eta)}.
\]

\begin{remark}\label{rem:quick_closing}
The most important property of the shift that we will use is that there exists some uniform $\ell$ such that if $\omega=(\omega_0,\ldots,\omega_n)$ is any valid finite string, then there exists a finite string $\eta$ of length at most $\ell$ such that $\omega\eta$ is valid, and, if repeated cyclically, defines a valid periodic word in $\Sigma$. This follows because transitivity of the SFT implies that there is a finite word going from any symbol to any other symbol.
\end{remark}

We now introduce linear cocycles. Suppose that $(X,d)$ is a compact metric space and that $\mc{E}$ is a vector bundle over $X$. If $\sigma\colon X\to X$ is a homeomorphism of $X$, then a \emph{linear cocycle} over $\sigma$ is a vector bundle map $\mc{A}\colon \mc{E}\to \mc{E}$ that fibers over $\sigma$. We write 
\begin{equation}\label{defn:cocycle_iterate}
\mc{A}(n,x)\colon \mc{E}_x\to \mc{E}_{\sigma^n(x)},
\end{equation}
for the associated linear map induced by the $n$th iterate of $\mc{A}$. 

In this paper, we consider cocycles over SFTs. In this context, a cocycle is defined by a map $A\colon \Sigma\to \GL(d,\R)$. We always consider the case when this map is H\"older continuous. In a natural way, the map $A$ defines a linear cocycle on the trivial bundle $\mc{A}\colon \Sigma\times \R^d\to \Sigma\times \R^d$. We note that for the applications in this paper that we only ever need to consider trivial bundles.  Calling both $A$ and $\mc{A}$ a ``cocycle" is a standard abuse of terminology. One of the main properties that we use is summarized in the following remark.

We will also need to consider ``H\"older continuous" vector bundles. 
While one could formulate the theory in extreme generality, one can also formulate the theory in a more basic way. For a trivial bundle $X\times \R^d$, it is easy to see what a H\"older section of this bundle is. Hence it is natural to define a \emph{H\"older vector bundle} over a metric space $X$ to be a vector bundle that embeds as a H\"older subbundle of $X\times \R^d$. More generally, one can define a H\"older bundle using H\"older  transition functions, although in this case it is more subtle to determine precisely the meaning of the bundle being $\beta$-H\"older because the composition of $\beta$-H\"older functions need not be $\beta$-H\"older. For a detailed discussion of H\"older vector bundles over metric spaces see \cite[Sec.~2.2]{bochi2019extremal}.

\begin{remark}\label{rem:holder_closeness}
If $A$ is a $\beta$-H\"older cocycle, then there exists $C$ such that if $\omega,\eta\in \Sigma$, and then if for all $\abs{i}\le n$, $\omega_i=\eta_i$, then 
\[
\|A(\omega)-A(\eta)\|<Ce^{-\beta n}, 
\]
which follows from the definition of H\"older continuity because $d(\omega,\eta)\le e^{-n}$.
\end{remark}

\subsection{Periodic data}
For a cocycle $\mc{A}$ over a homeomorphism $\sigma\colon X\to X$ on a vector bundle $\mc{E}$, one may consider the return map $\mc{A}(n,p)$ over any periodic point of period $n$. Typically, when one speaks of the ``periodic data" of this cocycle, one is referring to the map that associates each  periodic point $p$ of period $n$ with the conjugacy class of the matrix $\mc{A}(n,p)$ in $\GL(d,\R)$. In our case, we consider a slightly more general notion.

\begin{definition}\label{defn:constant_periodic_data}
    Given a list of real numbers $\lambda_1\ge \lambda_2\ge \cdots \ge \lambda_d$, we say that a cocycle $\mc{A}$ has periodic data with \emph{constant exponents} of type $(\lambda_1,\ldots,\lambda_d)$, if the following holds for each periodic point $p^n$ of period $n$:

    If we order the eigenvalues $\alpha_i$ of $\mc{A}(n,p^n)$ according to their moduli and write them according to their algebraic multiplicity (so that there are $d$ of them), then for each $1\le i\le d$ we have $\abs{\alpha_i}= e^{n\lambda_i}$.
\end{definition}

Note that this condition is more general than the condition of having all periodic data conjugate to the power of a single fixed matrix even in the case where the eigenspaces are $1$-dimensional.

    We will also consider periodic data that is tightly clustered around particular values. 

    \begin{definition}\label{defn:narrow_periodic_data}
    Given a list of numbers
$\lambda_1\ge \lambda_2\ge \cdots \ge \lambda_d$ and $\delta>0$, we say that the periodic data of a cocycle is \emph{$\delta$-narrow} around $\lambda_1\ge \cdots\ge \lambda_d$ if the following holds. For each periodic point $p^n$ of period $n$, if we order the eigenvalues $\alpha_i$ of $\mc{A}(n,p^n)$ as above, then for each $i$ we have
\[
e^{n(\lambda_i-\delta)}\le \abs{\alpha_i}\le e^{n(\lambda_i+\delta)}.
\]
\end{definition}
In particular, note that a cocycle with constant exponents is a cocycle that is $0$-narrow. This condition is similar to the condition of having ``narrow band spectrum" that appears in normal forms theory. See, for example, \cite{kalinin2020non}. Unlike in the normal forms theory however, we do not have any constraints due to resonances between the different $\lambda_i$.

\subsection{Singular values}

Recall that the singular values of a map $A\colon \R^d\to \R^d$ are the eigenvalues of the positive square-root of $A^*A$. We write the singular values of a linear operator $A\colon \R^d\to \R^d$ as 
\[
\sigma_1\ge \cdots\ge \sigma_d.
\]
If $A\colon \R^d\to \R^d$ is a linear map, then we denote by $A_k$ the induced map $A_k\colon \Lambda^k\R^d\to \Lambda^k\R^d$. A helpful fact that we use below is that 
\begin{equation}\label{eqn:singular_values_exterior_power}
\|A_k\|=\sigma_1(A)\cdots\sigma_k(A).
\end{equation}
We will use the following standard estimate on the perturbation of singular values. 
\begin{proposition}\label{prop:singular_values_perturbation}
\cite{stewart1979note}
Fix $d\ge 1$ then there exists $C>0$ such that if $A\colon \R^d\to \R^d$ is a linear map and $\sigma_1\ge \cdots\ge\sigma_d$ are the singular values of $A$, then if $E\colon \R^d\to \R^d$ is a perturbation and we write the singular values of $A+E$ as $\sigma_1'\ge \cdots\ge \sigma_d'$, then for $1\le i\le d$, 
\[
\abs{\sigma_i-\sigma_{i}'}\le C\|E\|.
\]
\end{proposition}

\subsection{Foliations}
In this paper we use the standard terminology for foliations. We say that a topological foliation of a smooth manifold $M$ has \emph{uniformly $C^{r}$ leaves} if the leaves of the foliation are $C^r$ and the $r$-jet of the foliation varies continuously. The same definition applies analogously for foliations that have uniformly $C^{1+\text{H\"older}}$ leaves. For more details, see \cite{pugh1997holder}. If we have two foliations $\mc{F}$ and $\mc{G}$ of a manifold $M$, then we say that a homeomorphism $h$ \emph{intertwines} $\mc{F}$ and $\mc{G}$ if it carries each leaf of $\mc{F}$ homeomorphically onto a leaf of $\mc{G}$. In other words, it is an isomorphism of topologically foliated manifolds.

\subsection{Anosov diffeomorphisms}

We say that a diffeomorphism $f\colon M\to M$ is Anosov if the tangent bundle $TM$ has a hyperbolic $Df$-invariant splitting. In this paper, we consider even finer splittings of $TM$. For an Anosov diffeomorphism $f$, we will often deal with the case that $f$ has a dominated splitting into many different subbundles, which we write as:
\begin{equation}\label{eqn:ordered_splitting}
E^{s,f}_{l}\oplus \cdots \oplus E^{s,f}_1\oplus E^{u,f}_1\oplus \cdots \oplus E^{u,f}_k,
\end{equation}
where the norm of $Df$ on each subbundle increases from left to right and $l=\dim E^s$ and $k=\dim E^u$, the dimensions of the stable and unstable subspaces of $f$, respectively.

For an Anosov automorphism $L\colon M\to M$, there is a dominated splitting into bundles corresponding to the distinct moduli of the eigenvalues of $L$. As in equation \eqref{eqn:ordered_splitting}, write $E^{s,L}_i$, $1\le i\le \dim E^s$ and $E^{u,L}_j$, $1\le j \le \dim E^u$ for the bundles in this $DL$-invariant splitting. From \cite[Sec.~2.1.1]{dewitt2021local}, it follows that there exists a $C^1$ small neighborhood $\mc{U}$ of $L$ such that the following holds. For each $f\in \mc{U}$, there is a $Df$-invariant splitting of $TM$ into H\"older continuous subbundles $E^{s,f}_i$, $1\le i\le \dim E^s$, and $E^{u,f}_j$, $1\le j \le \dim E^u$, such that each of these bundles is the same dimension as the corresponding bundle for $L$. In the case that $L$ is an Anosov automorphism of a torus, it follows that each of the bundles $E_i^{u,f}$ and $E_i^{s,f}$ integrates to a foliation. 
This follows from the perturbative theory of Hirsch-Pugh-Shub \cite{hirsch1977invariant}, which says that the weak bundle $E^{wu,f}_i=E_1^{u,f}\oplus \cdots \oplus E_i^{u,f}$ integrates to a uniformly $C^{1+\text{H\"older}}$ foliation, which we denote by $\mc{W}^{wu,f}_i$. Further, as every strong distribution integrates, $E_i^{u,f}\oplus \dots \oplus E_k^{u,f}$ also integrates to a foliation. We call this foliation $\mc{W}^{uu,f}_i$. By intersecting these two foliations, one obtains that $E_i^{u,f}$ integrates to a foliation with uniformly $C^{1+\text{H\"older}}$ leaves, which we denote by $\mc{W}^{u,f}_i$. Later we will see that a periodic data assumption is enough to conclude that these foliations exist even for maps that are not perturbations of linear maps.

In order to state our results in higher dimension, we need to define the Property $\mathcal A$ that appears in \cite[p.~647]{gogolev2008smooth}. Given a set $B$ and a foliation $\mc{F}$, we define:
\begin{equation}
\mc{F}(B)=\bigcup_{x\in B} \mc{F}(x).
\end{equation}

\noindent We say that a foliation is \emph{transitive} if it has a dense leaf. We say that a foliation is \emph{minimal} if every leaf is dense. We say that a foliation $\mc{F}$ of a manifold $M$ is \emph{tubularly minimal} if for each open set $B$, $\overline{\mc{F}(B)}=M$. It is shown by Gogolev \cite[Prop.~4]{gogolev2008smooth} that transitivity and tubular minimality for foliations of compact manifolds are equivalent. Consequently, the following definition is equivalent to the definition of Property $\mathcal A$ in \cite{gogolev2008differentiable}.

\begin{definition}
Let $\mc{U}$ be the neighborhood of $L\colon \mathbb{T}^d\to \mathbb{T}^d$ constructed above such that each $f\in \mc{U}$ has a dominated splitting corresponding to that of $L$ as well as foliations tangent to each bundle in this splitting. Then we say that an Anosov diffeomorphism $f$ has \emph{Property $\mathcal A$} if for any $1\le i\le \dim E^s-1$ and $1\le j\le \dim E^u-1$ the foliations $\mc{W}^{s,f}_i$ and $\mc{W}^{u,f}_j$ are transitive.
\end{definition}

\noindent While it is not known how ubiquitous Property $\mathcal A$ is, there are many diffeomorphisms that satisfy it. For instance, any irreducible hyperbolic toral automorphism satisfies Property $\mathcal A$~\cite[Prop.~6]{gogolev2008smooth}.
\section{Proof of the existence of a splitting}

In this section we prove Theorem \ref{thm:dominated_splitting_constant_data}.
We will use the following proposition, which is due to Kalinin \cite[Thm.~1.3]{kalinin2011livsic}.
\begin{proposition}\label{prop:maximum_norm_est_kalinin}
Suppose that $A$ is an $\GL(d,\R)$ cocycle over a transitive, invertible, subshift of finite type that has constant periodic data with exponents $\lambda_1\ge\lambda_2\ge \cdots\ge \lambda_d$. Then for all $\epsilon>0$ there exists $C_{\epsilon}$ such that for all $\omega\in \Sigma,n\in \N$,
\[
\|A^n(x)\|\le C_{\epsilon}e^{n(\lambda_1+\epsilon)}.
\]
\end{proposition}
The first step in the proof of Theorem \ref{thm:dominated_splitting_constant_data} is the following proposition. Note that the argument for the following proposition does not work if it is merely the case that the Lyapunov exponents of every measure are bounded away from zero rather than being close to constant.

\begin{proposition}\label{prop:lower_bound_norm}
Suppose that $A$ is a H\"older continuous $\GL(d,\R)$ cocycle with constant periodic data over a transitive shift with exponents 
\[
\lambda_1\ge \lambda_2\ge \cdots\ge  \lambda_{d-1}\ge \lambda_d.
\]
For every sufficiently small $\gamma>0$ and $\epsilon>0$, there exists $C>0$ such that if $p$ is a periodic point of period $n$ and $\gamma n\le i$, then 
\[
\|A^i(p)\|\ge Ce^{\gamma (\lambda_1-\epsilon)n}.
\]
\end{proposition}
\begin{proof}
Suppose that $p$ is a periodic point of period $n$. Then with $i$ as in the statement of the proposition we may write:
\begin{equation}\label{eqn:cocycle_two_pieces123}
A^n(p)=A^{n-i}(\sigma^i(p))A^i(p).
\end{equation}
Because the periodic data are constant we have that $\|A^n(p)\|\ge e^{\lambda_1 n}$. Because $i\ge \gamma n$,
we also have from Proposition \ref{prop:maximum_norm_est_kalinin} that for all $\epsilon>0$ there exists $C_{\epsilon}$ such that
\[
\|A^{n-i}(\sigma^i(p))\|\le C_{\epsilon} e^{(1-\gamma)(\lambda_1+\epsilon)n}. 
\]
Thus applying these estimates to \eqref{eqn:cocycle_two_pieces123} we see that:
\begin{align*}
e^{\lambda_1 n}\le \|A^i(p)\| C_{\epsilon}e^{(1-\gamma)(\lambda_1+\epsilon)n}.
\end{align*}
Thus
\begin{align}
\|A^i(p)\|&\ge C_{\epsilon}^{-1}e^{n\lambda_1-(1-\gamma)(\lambda_1+\epsilon)n}\\
&\ge C_{\epsilon}^{-1}e^{(\gamma\lambda_1-(1-\gamma)\epsilon )n}.
\end{align}
Up to this point, we have no relationship between $\epsilon$ and $\gamma$, hence replacing $\epsilon$ with $\gamma \epsilon/(1-\gamma)$, we obtain 
\[
\|A^i(p)\|\ge C_{\epsilon}e^{\gamma(\lambda_1-\epsilon)n},
\]
as desired.
\end{proof}

Next, by considering the action of the cocycle on exterior powers, we are able to use the previous proposition to gain information about other singular values.

\begin{proposition}\label{prop:singular_values_separate}
Suppose that $A$ is a H\"older continuous $\GL(d,\R)$ cocycle with constant periodic data over a transitive shift with exponents 
\[
\lambda_1\ge \lambda_2\ge \cdots \ge \lambda_{d-1}\ge \lambda_d.
\]
Then for every sufficiently small $\gamma,\epsilon>0$ and a fixed $\ell\in \N$, there exists $C_{\gamma,\epsilon}\ge 1$ such that if $p$ is a periodic point of period $n$ and $\gamma n\le i\le \gamma n+\ell$, then 
\begin{equation}\label{eqn:bound_on_sv_k}
C_{\gamma,\epsilon}^{-1}e^{\gamma n(\lambda_k-\epsilon)}\le \sigma_k(A^i(p))\le C_{\gamma,\epsilon}e^{\gamma n(\lambda_k+\epsilon)}.
\end{equation}
\end{proposition}

\begin{proof}
 We now proceed by induction on $k$ and verify equation \eqref{eqn:bound_on_sv_k}.
The base case when $k=1$ is immediate from Proposition \ref{prop:maximum_norm_est_kalinin} and Proposition \ref{prop:lower_bound_norm}. Now suppose that the result holds for $k-1$, then we show it holds for $k$.

We first check the upper bound in equation \eqref{eqn:bound_on_sv_k}.
We apply Proposition \ref{prop:lower_bound_norm} to the cocycle $A_{k-1}$ induced by $A$ on $\Lambda^{k-1} \R^d$. Hence for whatever sufficiently small $\gamma,\epsilon>0$ we choose there exists $C_{\gamma,\epsilon}>0$ such that for any periodic point $p$ of period $n$ and any $i\ge \gamma n$,  $i\le \gamma n+\ell$, we have that
\begin{equation}\label{eqn:k-1_power_lower}
\|A^i_{k-1}(p)\|\ge C_{\gamma,\epsilon} e^{\gamma n(\lambda_1+\cdots+\lambda_{k-1}-\epsilon)}.
\end{equation}
At the same time, by Proposition \ref{prop:maximum_norm_est_kalinin}, we have that for all $i\ge \gamma n$
\begin{equation}
\|A^i_k(p)\|\le C_{\epsilon}e^{i(\lambda_1+\cdots+\lambda_k+\epsilon)}.
\end{equation}
{In particular, for our fixed finite $\ell>0$ because $\abs{\gamma n-i}<\ell$, there exists $C'_{\epsilon}>0$ such that for any $\gamma n\le i\le \gamma n +\ell$, we have}
\begin{equation}\label{eqn:k-power_upper}
\|A^i_k(p)\|\le C_{\epsilon}'e^{\gamma n(\lambda_1+\cdots+\lambda_k+\epsilon)}.
\end{equation}
Recalling now that the norm of a linear map on the $k$-th exterior power is the product of the first $k$ singular values, we see that by dividing equation \eqref{eqn:k-power_upper} by equation \eqref{eqn:k-1_power_lower}, we see that there exists $C>0$ such that
\[
\sigma_{k}(A^i(p))\le Ce^{\gamma n (\lambda_k+2\epsilon)}.
\]

We now check the lower bound in Equation \eqref{eqn:bound_on_sv_k}.
First, we apply Proposition \ref{prop:lower_bound_norm} to get that 
\begin{equation}\label{eqn:upperboundsigma_1k}
\sigma_1\cdots \sigma_k=\|A^i_{k}(p)\|\ge e^{\gamma n(\lambda_1+\cdots+\lambda_{k}-\epsilon)}.
\end{equation}
But by the inductive hypothesis, we know that for $1\le j\le k-1$ that
\begin{equation}
\sigma_j(A^i(p))\le C e^{\gamma n (\lambda_j+\epsilon)}.
\end{equation}
Thus we obtain from equation \eqref{eqn:upperboundsigma_1k}, that for all $i\ge \gamma n$
\[
\sigma_k(A^i(p))\ge Ce^{\gamma n(\lambda_k-(k-1)\epsilon)}.
\]
This gives the second required bound and hence we have finished the induction. Since this is a finite induction we can redefine the initial $\eps$ so that all posited inequalities of Proposition~\ref{prop:singular_values_separate} are satisfied. This finishes the proof of the inequality \eqref{eqn:bound_on_sv_k}, so we are done by the first paragraph of the proof.
\end{proof}

We now use a shadowing argument to upgrade the above estimate for periodic points to an estimate along every trajectory.

\begin{proposition}\label{prop:constant_data_lower_bound_on_growth}
    Suppose that $A$ is a $\beta$-H\"older $\GL(d,\R)$ cocycle with constant periodic data over a transitive subshift of finite type $\Sigma$. 
  Suppose the periodic data has associated exponents
    \[
    \lambda_1\ge\lambda_2\ge \cdots\ge \lambda_{d-1}\ge \lambda_d.
    \]
    If $\lambda_k>\lambda_{k+1}$, then there exist $c,\tau>0$ such that for every $\omega\in \Sigma,n\in \N$, 
    \[
    \sigma_k(A^n(\omega))\ge ce^{\tau n}\sigma_{k+1}(A^n(\omega)).
    \]
\end{proposition}
\begin{proof}
To begin we show that there exist $\ell, C_1>0$ such that for each $\omega\in \Sigma$ and $n\in \N$ we may find a periodic orbit $p^n(\omega)$ of period $2n+\ell_n$ where $0\le \ell_n\le \ell$ such that 
    \begin{equation}\label{eqn:approximating_eqn}
        \|A(\sigma^i(\omega))-A(\sigma^i(p_n))\|\le C_1 e^{-\beta(n-i)}.
    \end{equation}
To do this, we apply Remark \ref{rem:quick_closing} and let $\ell$ be as in that remark. Let $\omega^n$ denote the finite word $(\omega_{-n},\ldots,\omega_n)$. Then there exists a finite word $\eta_n$ such that $\omega_n\eta_n$ is a valid finite word, and so by repeating this word cyclically we obtain an admissible element of $\Sigma$. We then let $p^n$ be the cyclic word chosen in this way so that for $\abs{i}\le n$, $p^n_i=\omega_i$. In particular, this implies that for $\abs{i}\le n$, that $d(\sigma^i(p_n),\sigma^i(\omega))\le e^{-(n-i)}$. From Remark \ref{rem:holder_closeness} equation \eqref{eqn:approximating_eqn} now follows.

For the rest of the proof fix some number $\mu>0$ so that for all $\omega\in \Sigma$ we have that: 
\begin{equation}\label{eqn:crude_upper_bound}
\|A(\omega)\|\le e^{\mu}.
\end{equation}

We now choose constants $0<\epsilon$, $0<\gamma<1$, $0<\kappa$, $0<\epsilon_0$ such that the following inequalities all hold:
\begin{align}
-\epsilon_0&<-\gamma\abs{\lambda_{k}-\epsilon}\label{eqn:ineq1}\\
-\epsilon_0&<-\gamma\abs{\lambda_{k+1}+\epsilon}\label{eqn:ineq2}\\
\mu\gamma-\beta(1-\gamma)+\gamma \kappa&<-\epsilon_0\label{eqn:ineq3}\\
0&<\lambda_k-\lambda_{k+1}-2\epsilon\label{eqn:ineq4}.
\end{align}
To see that we may arrange for all of these inequalities to hold simultaneously, note that inequalities \eqref{eqn:ineq1}, \eqref{eqn:ineq2}, and \eqref{eqn:ineq3} are equivalent to
\begin{equation}
0<\gamma<\min\left\{\frac{\epsilon_0}{\abs{\lambda_{k}-\epsilon}},\frac{\epsilon_0}{\abs{\lambda_{k+1}-\epsilon}}, \frac{-\epsilon_0+\beta}{\mu+\beta+\kappa}\right\}.
\end{equation}
For sufficiently small $\epsilon_0,\epsilon>0$ it is straightforward to find a choice of  $0<\gamma<1$ and $0<\kappa$ such that first three inequalities hold. Ensuring that the final inequality holds as well is straightforward because it holds for all sufficiently small $\epsilon>0$.

In the rest of the proof, we will study the difference between the cocycle along the orbit of $\omega\in \Sigma$ and the cocycle along the orbit of $p^{n}$. Thus define a matrix $E_{n,i}(\omega)$ by
\begin{equation}
E_{n,i}(\omega)=A(\sigma^i(\omega))-A(\sigma^i(p_n(\omega))).
\end{equation}
For the rest of the proof, we will suppress the argument $\omega$ as all estimates are independent of the particular point under consideration.

Consider
\begin{align}\label{eqn:expression_for_product}
\prod_{i=0}^{\lfloor\gamma n\rfloor} A(\sigma^i(\omega))=\prod_{i=0}^{\lfloor \gamma n\rfloor} \left[ A(\sigma^i(p^n))+E_{i,n}\right].
\end{align}

We now need a claim.
\begin{claim}\label{claim:binom_estimate}
For any fixed $\kappa>0$ that there exists $C_{\kappa}>0$ such that ${n\choose k}\le C_{\kappa}e^{nk\kappa}$.
\end{claim}
\begin{proof}
    To begin, use the standard estimate that ${n\choose k}\le n^k/k!$. Then to see that $n^k/k!\le C_{\kappa}\exp(\kappa n k)$ take logarithms. Thus we are checking that $k\ln (n)-\ln(k!)\le \kappa n k +\ln(C_{\kappa})$ for some constant $C_{\kappa}$. As long as $\kappa n>\ln(n)$ this inequality holds with $\ln(C_{\kappa})=0$. Hence we need only consider $n$ less than some fixed $N_{\kappa}$. But for the finite number of $n\le N_{\kappa}$, the $\ln(k!)$ term grows faster than the $k\ln(n)$ term. Thus for each such $n$, there are only finitely many $k$ such that the inequality does not hold when $\ln(C_{\kappa})=0$. Thus we see that for $C_{\kappa}=1$, there are only finitely many integer pairs $(n,k)$ such that the estimate ${n\choose k}\le C_{\kappa}e^{nk\kappa}$ fails. Hence by increasing $C_{\kappa}$, we obtain the result.
    \end{proof}

If we expand the right hand side of \eqref{eqn:expression_for_product} as a product, it becomes the sum of $2^{\lfloor \gamma n\rfloor}$ different terms. We then group the terms involving exactly $k$ copies of the error terms  $E_{i,n}$. Call the sum of all terms with $k$ factors of the form $E_{i,n}$ by $B_{k,n}$. Then we may write:
\[
\prod_{i=0}^{\lfloor \gamma n\rfloor} \left[ A(\sigma^i(p^n))+E_{i,n}\right]=\prod_{i=0}^{\lfloor \gamma n\rfloor}  A(\sigma^i(p^n))+\sum_{k=1}^{\lfloor \gamma n\rfloor} B_{k,n}
\]

We now estimate the terms $B_{k,n}$. The term $B_{k,n}$ is the sum of ${\lfloor \gamma n\rfloor\choose k}$ summands each involving $k$ of the terms $E_{i,n}$ where $0\le i\le \lfloor \gamma n\rfloor$. By equation \eqref{eqn:approximating_eqn}, because $1\le i\le \lfloor \gamma n\rfloor$, it follows that $\|E_{i,n}\|\le C_1e^{-\beta(1-\gamma)n}$. Each individual summand contributing to $B_{k,n}$ is the product of $k$ $E_{i,n}$ terms as well as $\lfloor \gamma n\rfloor-k$ terms of the form $A(\sigma^i(p^n))$. Thus as $B_{k,n}$ is the sum of ${\lfloor \gamma n\rfloor\choose k}$ such terms, it then follows from equations \eqref{eqn:crude_upper_bound} and \eqref{eqn:approximating_eqn} that there exists $C_2$ such that:
\[
\|B_{k,n}\|\le C_2{\lfloor n\gamma\rfloor\choose k}e^{-k\beta(1-\gamma)n}e^{\mu (\gamma n-k)},
\]
for some $C_2$ depending only on $\Sigma$ and the $\beta$-H\"older norm of $A$. From Claim \ref{claim:binom_estimate}, we see that for any $\kappa>0$, there exists $C_{\kappa}$ such that:
\[
\|B_{k,n}\|\le C_2C_{\kappa}e^{\gamma n k \kappa} e^{-k\beta(1-\gamma)n}e^{\mu (\gamma n-k)}.
\]
Thus 
\begin{align}
\left\|\sum_{k=1}^{\lfloor \gamma n\rfloor} B_{k,n}\right\|&\le \sum_{k=1}^{\lfloor \gamma n\rfloor}C_2C_{\kappa}e^{\gamma n k \kappa} e^{-k\beta(1-\gamma)n}e^{\mu (\gamma n-k)}\\
&\le C_2C_{\kappa}e^{\mu \gamma n}\sum_{k=1}^{\infty}e^{-(\mu+\beta(1-\gamma)n-\gamma n\kappa)k}\\
&\le C_2C_{\kappa}e^{\mu \gamma n}\frac{e^{-(\mu+\beta(1-\gamma)n-\gamma n\kappa)}}{1-e^{-(\mu+\beta(1-\gamma)n-\gamma n\kappa)}}.\label{eqn:last_upper_bound_a}
\end{align}

Thus from our choice of constants, \eqref{eqn:ineq3} is satisfied, so it immediately follows from inequality \eqref{eqn:last_upper_bound_a} that there exists $D_1>0$ such that 
\begin{equation}
    \left\|\sum_{k=1}^{\lfloor \gamma n\rfloor} B_{k,n}\right\|\le D_1 e^{-\epsilon_0n}.
\end{equation}

Thus we have obtained that:
\begin{equation}\label{eqn:final_inequality}
\prod_{i=0}^{\lfloor\gamma n\rfloor}A(\sigma^i(\omega))=E_n+\prod_{i=0}^{\lfloor\gamma n\rfloor} A(\sigma^i(p^n)),
\end{equation}
where $\|E_n\|\le D_1e^{-\epsilon_0 n}$, where $D_1$ depends only on the cocycle and our choices of the parameters $\gamma,\kappa,\epsilon,\epsilon_0$. 

Note now that for each $m\in \mathbb{N}$ that there exists some $n$ such that $\lfloor \gamma n\rfloor=m$. Thus there exists $D_2>0$ such that for each $n\in \mathbb{N}$ there exists $E_n'$ such that:
\begin{equation}
\prod_{i=0}^{n}A(\sigma^i(\omega))=E_n'+\prod_{i=0}^{n} A(\sigma^i(p^{\lceil\gamma^{-1}n\rceil})),\text{ where } \|E_n'\|\le D_2 e^{-\epsilon_0\gamma^{-1}n}.
\end{equation}

Thus it follows from Proposition \ref{prop:singular_values_perturbation}, that for all $n\in \N$,
\begin{align}
\abs{\sigma_k(A^n(\omega))-\sigma_k(A^n(p^{\lceil\gamma^{-1}n\rceil}))}&\le D_2e^{-\epsilon_0\gamma^{-1}n}\label{eqn:closeness_of_sigma_k},\\ 
\abs{\sigma_{k+1}(A^n(\omega))-\sigma_{k+1}(A^n(p^{\lceil\gamma^{-1}n\rceil}))}&\le D_2e^{-\epsilon_0\gamma^{-1}n}.\label{eqn:closeness_of_sigma_k_plus_1}
\end{align}

And from Proposition \ref{prop:singular_values_separate} it follows that there exists $C_{\gamma,\epsilon}>0$ depending only on $\gamma,\epsilon$ and the cocycle such that for all $n\in \N$,
\begin{align}
C_{\gamma,\epsilon}^{-1}e^{ n(\lambda_k-\epsilon)}&\le \sigma_k(A^n(p^{\lceil\gamma^{-1}n\rceil}))\le C_{\gamma,\epsilon}e^{ n(\lambda_k+\epsilon)}\label{eqn:bounds_on_k},\\
C_{\gamma,\epsilon}^{-1}e^{ n(\lambda_{k+1}-\epsilon)}&\le \sigma_{k+1}(A^n(p^{\lceil\gamma^{-1}n\rceil}))\le C_{\gamma,\epsilon}e^{ n(\lambda_{k+1}+\epsilon)}\label{eqn:bounds_on_k_plus_1}.
\end{align}

Thus by\eqref{eqn:closeness_of_sigma_k}, \eqref{eqn:closeness_of_sigma_k_plus_1}, \eqref{eqn:bounds_on_k}, \eqref{eqn:bounds_on_k_plus_1}, we see that
\begin{align}
\sigma_k(A^n(\omega))&\ge C_{\gamma,\epsilon}^{-1}e^{ n(\lambda_k-\epsilon)}-D_2e^{-\epsilon_0\gamma^{-1}n}\label{eqn:aa}\\
\sigma_{k+1}(A^n(\omega))&\le C_{\gamma,\epsilon}e^{ n(\lambda_{k+1}+\epsilon)}+D_2e^{-\epsilon_0\gamma^{-1}n}\label{eqn:bb}.
\end{align}
From inequalities \eqref{eqn:ineq1} and \eqref{eqn:ineq2} we have that:
\begin{equation}\label{eqn:ineq10}
-\epsilon_0\gamma^{-1} < \lambda_{k+1}+\epsilon\text{ and } -\epsilon_0\gamma^{-1}< \lambda_k-\epsilon.
\end{equation}
Thus from the inequalities \eqref{eqn:aa}, \eqref{eqn:bb}, and \eqref{eqn:ineq10} we see that there exists $D_3$ such that
\begin{align}
\sigma_k(\omega)\ge D_{3}^{-1}e^{n(\lambda_k-\epsilon)}\\
\sigma_{k+1}(\omega)\le D_{3}e^{n(\lambda_{k+1}+\epsilon)}.
\end{align}
From this it is immediate that
\begin{equation}
\sigma_k(\omega)\ge D_{3}^{-2} e^{n(\lambda_k-\lambda_{k+1}-2\epsilon)}\sigma_{k+1}(\omega),
\end{equation}
which is the needed conclusion because $\lambda_k-\lambda_{k+1}-2\epsilon>0$ by \eqref{eqn:ineq4}.
\end{proof}

We can now deduce the main theorem of this paper. 

\begin{proof}[Proof of Theorem \ref{thm:dominated_splitting_constant_data}]

From Proposition \ref{prop:constant_data_lower_bound_on_growth}, the cocycle satisfies the criterion of Bochi--Gourmelon (Theorem \ref{thm:yoccoz_criterion}), and hence has a dominated splitting, so we are done.
\end{proof}

\subsection{Narrow periodic data}

In this section, we show that our result that produces dominated splittings works even when the periodic data is only close to being constant. For sufficiently narrow periodic data, one can still obtain a dominated splitting.

In this section, we sketch the proof of Theorem \ref{thm:narrow_data} for $\delta$-narrow periodic data of $\GL(2,\R)$ cocycles. The proof in the general case is quite similar but more complicated. So, suppose that $A$ is a $\GL(2,\R)$ cocycle with $\delta$-narrow periodic data around $\lambda_1>\lambda_2$. We will conduct the rest of the proof under the additional assumption that $\lambda_1>0>\lambda_2$. The analysis would be similar without this assumption. The proof proceeds along the same lines as the proof for constant periodic data except now there is slightly more information to keep track of.

To begin, as in Proposition \ref{prop:lower_bound_norm}we obtain that for all $\epsilon,\gamma>0$ there exist $C,\ell>0$ such that for any $\gamma n\le i\le \gamma n+\ell$ and any periodic point $p$ of period $n$,
\begin{equation}\label{eqn:lower_bound_sigma_1_a}
\sigma_1(A^i(p))\ge C_{\gamma,\epsilon} e^{
((\lambda_1-\delta)-(1-\gamma)(\lambda_1+\delta+\epsilon))n}.
\end{equation}
For the induced cocycle $A_2$ on the exterior power $\Lambda^2\R^2$, the periodic data is $2\delta$-narrow around $\lambda_1+\lambda_2$. Thus as before, by Proposition \ref{prop:maximum_norm_est_kalinin}, we obtain that for all $\epsilon>0$, there exists $C_{\epsilon}$ such that 
\begin{equation}\label{eqn:lower_bound_sigma_1_b}
\sigma_1(A_2^i(p))=\sigma_1(A^i(p))\sigma_2(A^i(p))\le C_{\epsilon}e^{i(\lambda_1+\lambda_2+2\delta+\epsilon)}.
\end{equation}
Keeping in mind that $i$ and $\gamma n$ differ by at most $\ell$, we take the quotient of equations \eqref{eqn:lower_bound_sigma_1_a} and \eqref{eqn:lower_bound_sigma_1_b} to obtain that for all $\epsilon,\gamma>0$, there exists $D_{\epsilon,\gamma}$ such that for $\gamma n\le i\le \gamma n +\ell$,
\[
\sigma_2(A^i(p))\le D_{\epsilon,\gamma}e^{n((\lambda_2+\delta)\gamma+2\delta+\epsilon)}.
\]
Suppose we now carry out the argument in Proposition \ref{prop:constant_data_lower_bound_on_growth}. Then we obtain---in the notation of the proof of that proposition---that for an arbitrary point $\omega$ that for each $n$, if $\gamma n\le i\le \gamma n+\ell$, then
\begin{align}
\sigma_1(A^i(\omega))&\ge C_{\gamma,\epsilon}e^{n((\lambda_1-\delta)-(1-\gamma)(\lambda_1+\delta+\epsilon))}-D_2e^{-\epsilon_0 n}\label{eqn:lower_bound_k_1}\\
\sigma_2(A^i(\omega))&\label{eqn:upper_bound_k_2}\le D_{\epsilon,\gamma}e^{n((\lambda_2+\delta)\gamma+2\delta+\epsilon)}+D_2e^{-\epsilon_0 n}
\end{align}
where 
\[
-\epsilon_0=\mu\gamma -\beta(1-\gamma)+\gamma\kappa<0,
\]
due to the same estimates as in the proof of Proposition \ref{prop:constant_data_lower_bound_on_growth}.

In order to proceed, it suffices to show that the $D_2$ terms in equation \eqref{eqn:lower_bound_k_1} and \eqref{eqn:upper_bound_k_2} are negligible relative to the first terms in those equations. The two constraints we obtain are, respectively,
\begin{align}
\mu\gamma -\beta(1-\gamma)+\gamma\kappa&<\lambda_1-\delta-\lambda_1-\delta-\epsilon+\gamma\lambda_1+\gamma\delta+\gamma\epsilon,\\
\mu\gamma -\beta(1-\gamma)+\gamma\kappa&<(\lambda_2+\delta)\gamma+2\delta+\epsilon.
\end{align}
Thus we can find $\gamma$ that satisfies this equation as long as 
\begin{equation}\label{eqn:upper_bound_gamma_from_D_2}
\gamma<\min\left\{\frac{\beta-2\delta-\epsilon}{\mu-\lambda_1-\delta-\epsilon+\beta+\kappa},\frac{\beta+2\delta
+\epsilon}{\mu-\lambda_2-\delta+\beta+\kappa}\right\}.
\end{equation}

As long as $\gamma$ satisfies inequality \eqref{eqn:upper_bound_gamma_from_D_2}, then we may disregard the lower order $D_2$ terms in \eqref{eqn:lower_bound_k_1} and \eqref{eqn:upper_bound_k_2}. Hence we obtain the bound:
\begin{equation}
\sigma_1(A^i(\omega))\ge D_3e^{n[((\lambda_1-\delta)-(1-\gamma)(\lambda_1+\delta+\epsilon))-((\lambda_2+\delta)\gamma+2\delta+\epsilon)]}\sigma_2(A^i(\omega)).
\end{equation}
Thus we will be able to conclude if we know that 
\begin{equation}
((\lambda_1-\delta)-(1-\gamma)(\lambda_1+\delta+\epsilon))-((\lambda_2+\delta)\gamma+2\delta+\epsilon)>0,
\end{equation}
which holds as long as 
\begin{equation}\label{eqn:second_constraint_on_gamma}
\gamma\ge \frac{4\delta+2\epsilon}{\lambda_1-\lambda_2+\epsilon}.
\end{equation}
Thus there will exist $\gamma$ satisfying \eqref{eqn:upper_bound_gamma_from_D_2} and \eqref{eqn:second_constraint_on_gamma} as long as  
\begin{equation}
    \frac{4\delta+2\epsilon}{\lambda_1-\lambda_2+\epsilon}<\min\left\{\frac{\beta-2\delta-\epsilon}{\mu-\lambda_1-\delta-\epsilon+\beta+\kappa},\frac{\beta+2\delta
+\epsilon}{\mu-\lambda_2-\delta+\beta+\kappa}\right\}.
\end{equation}

As several of the parameters may be taken to be arbitrarily small, we can optimize this estimate further. By Kalinin's work on the Livsic theorem \cite[Thm.~1.3]{kalinin2011livsic}, we take $\mu$ as close to $\lambda_1+\delta$ as we like. Further we may always pick $\epsilon$ and $\kappa$ arbitrarily close to zero. Thus we see that for a $\beta$-H\"older $\GL(2,\R)$ cocycle with $\delta$-narrow periodic data around $\lambda_1> \lambda_2$, there is a dominated splitting of $A$ if 
\begin{equation}
    \frac{4\delta}{\lambda_1-\lambda_2}<\min\left\{\frac{\beta-2\delta}{\beta},\frac{\beta+\delta
}{\lambda_1-\lambda_2+\beta}\right\}
\end{equation}
Note that this inequality is always satisfied if $\delta$ is sufficiently small. 

Note that when $\lambda_1-\lambda_2$ is large, that we may take $\delta$ to be large relative to $\beta$. It is useful to note that the case where $\lambda_1-\lambda_2$ is close to $0$ is not particularly novel because these cocycles are actually fiber bunched and were already analyzed by the work of Velozo Ruiz \cite{velozo2020characterization}. Note that by following along the lines above one may similarly compute how small $\delta$ must be for a cocycle with $\delta$-narrow spectrum around a list of numbers $\lambda_1\ge \cdots\ge\lambda_d$ to necessarily have a dominated splitting.

It is interesting to relate the prior discussion to a question due to Katok that has been studied by Travis Fisher and the second author~\cite{fisher2006some, gogolev2010diffeomorphisms}. One way to formulate it is as follows.
\begin{quest}
Let $L$ be an Anosov automorphism with only two Lyapunov exponents $-\lambda<\lambda$. Assume that $f$ is a volume preserving diffeomorphism that is H\"older conjugate to $L$. Is $f$ also Anosov?
\end{quest}

Hence the question is asking for existence of a dominated splitting for $Df$, which can naturally be considered as a H\"older continuous cocycle over $f$, which is a uniformly hyperbolic homeomorphism. In the following discussion, note that because we have assumed that $f$ is volume preserving that we are now essentially studying an $\SL(2,\R)$ cocycle, rather than a $\GL(2,\R)$ cocycle. In the two dimensional setting, this substantially simplifies the analysis. The H\"older continuity of the conjugacy and its inverse is essentially equivalent to a narrowness condition on the periodic data. In particular, by studying the rate at which points approach a periodic point, we can deduce that if $h$ is $C^{\theta}$ and $h^{-1}$ is $C^{\omega}$, then the positive exponent of the periodic points of $Df$ lies in $(\lambda-\delta,\lambda+\delta)$ for some $\delta$ satisfying
\begin{equation}\label{eqn:theta_omega_for_delta_pm}
\theta\le \frac{\lambda-\delta}{\lambda},\,\,\,\,\,\omega\le \frac{\lambda}{\lambda+\delta}.
\end{equation}

\clearpage

In~\cite[Thm.~1]{gogolev2010diffeomorphisms}, examples of diffeomorphisms conjugate to Anosov diffeomorphisms but are not Anosov are constructed. In fact, quantitative bounds are obtained: for each $\epsilon>0$, there exists a diffeomorphism $f$ H\"older conjugate to Anosov such that $\theta\omega=1/8-\eps$, $Df$ does not have a dominated splitting, and $f$ is not Anosov. Necessarily, such an $f$ cannot have narrow spectrum.
On the other hand Fisher's result~\cite[Thm.~5.1]{fisher2006some}, \cite[Thm.~2]{gogolev2010diffeomorphisms}, says that if $\theta\omega>1/2$ then $f$ is Anosov. 

As regularity of the conjugacy implies a particular $\delta$-narrow spectrum condition, it is interesting to see how much regularity of the conjugacy is needed to show that $Df$ has a hyperbolic splitting by pulling $Df$ back to a cocycle over $L$ and applying the techniques described above.

In the case of $\SL(2,\R)$ cocycles, producing a dominated splitting is particularly easy because one only needs to show that $\sigma_1(A^i(\omega))$ is growing exponentially fast, which results in fewer inequalities to check. In the context of the proof just outlined for $\GL(2,\R)$ cocycles, we just need to see that the right hand side of equation \eqref{eqn:lower_bound_k_1} is dominated by the first term, and that this first term is, in fact, growing exponentially. This leads to two inequalities corresponding to these two conditions:
\begin{align}
(\lambda-\delta)-\lambda-\delta-\epsilon+\gamma\lambda+\gamma\delta&>(\mu+\epsilon)\gamma-\lambda\beta(1-\gamma)\\
(\lambda-\delta)-\lambda-\delta-\epsilon+\gamma\lambda+\gamma\delta&>0.
\end{align}
Note that the first inequality has an additional factor of $\lambda$ on the final term that did not appear above: this is because the shadowing argument uses the hyperbolicity of the base dynamics, which for the shift is normalized so that $\lambda=1$.

A dominated splitting will exist if there is $0<\gamma<1$ solving these two inequalities. Using as before that we can choose $\epsilon$ arbitrarily small and use $\lambda_1+\delta$ as an upper bound for $\mu$, we see that there will exist some $0<\gamma<1$ satisfying this pair of inequalities as long as 
\begin{equation}\label{eqn:ineq_volume_preserving}
\frac{2\delta}{\lambda+\delta}<\frac{\lambda\beta-2\delta}{\lambda\beta}.
\end{equation}

We now determine when the inequality \eqref{eqn:ineq_volume_preserving} holds in terms of $\omega$ and $\theta$. As the cocycle we are studying is the pullback of $Df$ by $h$, this cocycle is $\theta$ H\"older, so $\beta=\theta$ in \eqref{eqn:ineq_volume_preserving}. The following bounds then follow easily from \eqref{eqn:theta_omega_for_delta_pm}:
\begin{equation}
    \frac{2\delta}{\lambda+\delta}\le 1-\omega\theta\text{ and } \frac{1}{\omega\theta}<\frac{\lambda\beta-2\delta}{\lambda \beta}.
\end{equation}
Thus we see that there is a dominated splitting as long as 
\[
1-\omega\theta<\frac{1}{\theta \omega}
\]
Thus there is a dominated splitting as long as $\omega\theta$ is larger than the positive root of $x^2+x-1$, which is $(\sqrt{5}-1)/2>.618$.
This is a weaker conclusion than Fisher's $\omega\theta>1/2$. We note, however, that our approach is through general results on cocycles while Fisher's proof is very geometric and relies on Ma\~n\'e's characterization of hyperbolicity, so it's not surprising that our result is a bit weaker.
 
An advantage of our approach is that it can also be  applied to recover splittings other than the splitting into stable and unstable subspaces from a sufficiently regular conjugacy. 

\begin{corollary}
Suppose that an Anosov automorphism $L$ has a dominated splitting of index $k$ of its unstable bundle. Then there exist $\theta_0<1$ and $\omega_0<1$ such that the following holds. If $f$ is a $C^{2}$ diffeomorphism that is H\"older conjugate to $L$ by a conjugacy $h$ such that $h$ is $\theta$-H\"older and $h^{-1}$ is $\omega$-H\"older with $\theta>\theta_0$ and $\omega>\omega_0$, then $f$ is also Anosov and has a dominated splitting of index $k$ of its unstable bundle.
\end{corollary}

\subsection{Cocycles over general hyperbolic systems}

In the case of cocycles over general hyperbolic systems the same results hold. In principle, there are two ways to obtain such results. One way is to code the system by using the machinery of Markov partitions and then formally deduce the result in the general setting from the symbolic one. There will be some quantitative loss due to passing through a H\"older coding map. Since the argument only relies on shadowing, one may also redo the entire argument presented above. In our view, this is a cleaner approach. As the argument is so similar, we omit it but we now describe the setting. The only advantage that symbolic setting has provided is expositional as shadowing takes a particularly simple form for SFTs. The properties we require in the general setting are the same as those considered in \cite{kalinin2011livsic} because the results of that paper are used in the proof. The first is essentially the definition of exponential shadowing.

\begin{definition}\label{def:exp_shadowing}
We say that two orbit segments $x,f(x),\ldots f^n(x)$ and $p,f(p),\ldots$ $f^n(p)$ are \emph{exponentially $\delta$-close with exponent $\lambda>0$} if for every $i=0,\ldots,n$, we have 
\begin{equation}
d(f^i(x),f^i(p))\le \delta\exp(-\lambda\min\{i,n-i\}).
\end{equation}
\end{definition}

\begin{definition}\label{defn:closing_property}
(Closing Property) We say that a homeomorphism $f\colon (X,d)\to (X,d)$ of a metric space satisfies the \emph{closing property} if there exist $c,\lambda,\ell>0$ such that for any $x\in X$ and $n>0$ there exists $0\le j\le \ell$ and a point $p\in X$ such that $f^{n+j}(p)=p$ and the orbit segments $x,f(x),\ldots,f^n(x)$ and $p,f(p),\ldots,f^n(p)$ are exponentially $\delta=cd(x,f^n(x))$ close with exponent $\lambda$. Further, we assume that there exists a point $y\in X$ such that for every $i=0,\ldots,n$
\[
d(f^i(p),f^i(y))\le \delta e^{-\lambda i}\text{ and }d(f^i(y),f^i(x))\le \delta e^{-\delta(n-i)}.
\]
\end{definition}
\noindent It is well-known that locally maximal hyperbolic sets satisfy these properties. As a special case, the above apply to Anosov diffeomorphisms.  See for instance \cite[Ch.~18]{katok1995introduction}.

Using these definitions, we formulate a more general version of our main technical result which should be useful for future applications. We remark, however, that this result is not needed for the applications in this paper because the unstable and stable bundles of an Anosov diffeomorphism of a torus are always trivial.
\begin{theorem}\label{thm:non-trivial_bundle}
Suppose that $f\colon X\to X$ is a homeomorphism of a compact metric space satisfying the closing property in Definition \ref{defn:closing_property}. Let $\lambda_1\ge\cdots\ge \lambda_d$ be a non-increasing list of real numbers with $\lambda_k>\lambda_{k+1}$ for some $k$, and let $\mc{E}$ be a H\"older vector bundle over $X$ in the sense of \cite[Sec.~2.2]{bochi2019extremal}. Then for any $\beta>0$ there exists $\delta>0$ such that any $\beta$-H\"older linear cocycle $\mc{A}\colon\mc{E}\to \mc{E}$ over $f$ that has $\delta$-narrow spectrum around $\lambda_1\ge \cdots\ge \lambda_d$ has a dominated splitting of index $k$.
\end{theorem}

\section{Application to Anosov Diffeomorphisms}

From the main technical result we can now deduce our periodic data rigidity theorems. We remark that this section is mostly an explanation of how to combine our  Theorem \ref{thm:dominated_splitting_constant_data} with existing results. That said there are other novelties such as Observation \ref{obs:c1_conj}, which turns local $C^1$ results into global $C^1$ results.

\subsection{Rigidity of non-linear Anosov diffeomorphisms}

We now turn to the proof of our non-linear rigidity theorem, Theorem \ref{thm:non_linear_rigidity}. We organize the proof into steps with Lemma \ref{lem:dominated_splitting_exists} being the main novel step. The remaining steps are already present in~\cite{gogolev2008differentiable, gogolev2008smooth}. Nonetheless, we discuss them to show how they fit together. To begin, we must recover the dominated splitting.

\begin{lemma}\label{lem:dominated_splitting_exists}
Suppose that $L$ is an Anosov automorphism of $\mathbb{T}^d$. Let $\epsilon>0$ be a sufficiently small number, then there is a $C^1$ neighborhood $\mc{U}$ of $L$ such that if $f\in \mc{U}$, and $g$ is any $C^2$ Anosov diffeomorphism with the same periodic data as $f$, then $T\mathbb{T}^d$ has a continuous $Dg$-invariant splitting into bundles $E^{u,g}_i$, $E^{s,g}_j$ corresponding to the bundles of $L$. Further, there is a H\"older metric on $E^{u,g}$ such that if $\lambda_1\ge\cdots\ge \lambda_m>0$ are the unstable Lyapunov exponents of $L$, for $v\in E_i^{u,g}$, we have 
\begin{equation}\label{eqn:narrow_bnd_spectrum}
e^{\lambda_i-\epsilon}\|v\|\le \|Dgv\|\le e^{\lambda_i+\epsilon}\|v\|.
\end{equation}
The analogous statement holds for $E^{s,g}$ as well.
\end{lemma}

\begin{proof}[Proof.]

Fix a small $\delta'>0$ and let $\cU'$ be a $\delta'$-small $C^1$ neighborhood of $L$. Let $f\in\cU'$ and assume that $g$ is an Anosov diffeomorphism with the same periodic data as $f$.

\begin{lemma}\label{lem:holder_control} Any diffeomorphism $g$ with the same periodic data as $f$ has the H\"older exponent of its stable and unstable bundles, and these constants depend only on $L$ and $\delta'$ and hence are uniform in $g$.
\end{lemma}
\begin{proof} As is standard (see \cite[Thm~2.3]{dewitt2021local}), for all $f\in\cU'$, being $\delta'$-close to $L$, must satisfy the same exponential bounds 
\[
C^{-1}\mu_1^n\le\|D^sf^n\|\le C^{}\mu_2^n,\,\,\,\, C^{-1}\nu_1^n\le\|D^uf^{-n}\|\le C^{}\nu_2^n,\,\,\, n\ge 1,
\]
where $\mu_i\in(0,1)$ and $\nu_i\in(0,1)$ only depend on $L$ and $\delta'$. Since $f$ and $g$ have the same periodic data we have the same bounds for $g$ over all periodic points $p=g^np$:
$$
C^{-1}\mu_1^n\le\|D^s_p g^n\|\le C^{}\mu_2^n,\,\,\,\, C^{-1}\nu_1^n\le\|D^u_pg^{-n}\|\le C^{}\nu_2^n.
$$
Now we can apply~\cite[Theorem~1.3]{kalinin2011livsic} to conclude that all $g$ also satisfy exponential bounds for all points, and uniformly in $g$. Namely for any $\eps>0$ there exists $C_\eps>0$ such that
$$
C_\eps^{-1}(\mu_1-\eps)n\le\|D^sg^n\|\le C_\eps^{}(\mu_2+\eps)^n,\,\,\,\, C_\eps^{-1}(\nu_1-\eps)^n\le\|D^ug^{-n}\|\le C_\eps^{}(\nu_2+\eps)^n.
$$
It is well known that the unstable subbundle of such $g$ is H\"older with a uniform constant and uniform H\"older exponent given by
$$
\min\left\{1, \frac{\log(\mu_2+\eps)+\log(\nu_2+\eps)}{\log{(\mu_1-\eps)}}\right\}.
$$
This H\"older property of subbundles is due to Anosov and the explicit expression for the exponent was given by Hirsch and Pugh~\cite{hirsch1970stable}.
\end{proof}

Now we note that the unstable subbundle of $g$ is trivial because $g$ is conjugate to $L$  \cite[Prop.~35]{dewitt2022periodic}. Hence the unstable differential $D^ug$ defines a cocycle over $g$ which is H\"older uniformly in $g$. We note that the above lower bound on the H\"older exponent could only improve if $\delta'$ is chosen to be even smaller. We are in the position to apply Theorem~\ref{thm:narrow_data} to conclude that there exists a $\delta$ such that if the periodic data of $g$ is $\delta$-narrow than $g$ admits a dominated splitting corresponding to the splitting for $L$. Recall that $\delta>0$ depends only on spectrum of $L$ and on (the lower bound on) the H\"older exponent of $D^ug$. 

So now we can choose an even smaller $C^1$ neighborhood $\cU$ of $L$ such that periodic data of $f\in \cU$ (and hence of $g$ as well) is $\delta$-narrow. Then Theorem~\ref{thm:narrow_data} indeed applies to such $g$ and we obtain a dominated splitting for $g$ which matches the dominated splitting for $f$.
\end{proof}

Next, we will use the following lemma to obtain that weak foliations exist even globally.

\begin{lemma}\label{lem:weak_foliations_exist}
Let $L$ be an Anosov automorphism of $\mathbb{T}^d$. Then there exists a $C^1$ open neighborhood $\mc{U}$ of $L$ in $\Diff^{2}(\mathbb{T}^d)$, such that for any $f\in \mc{U}$, the following holds. If $g$ is any $C^{2}$ Anosov diffeomorphism with the same periodic data as $f$, then $E^{u,g}$ has a $Dg$ invariant splitting into bundles $E_i^{u,g}$ corresponding to the bundles of $L$. Further, the weak flag of bundles $E^{wu,g}_i=E_1^{u,g}\oplus \cdots \oplus E_i^{u,g}$ is uniquely integrable to a weak foliation with uniformly $C^{1+\text{H\"older}}$ leaves, which we denote by $\mc{W}^{wu,g}_i$. Further, a conjugacy $h$ between $f$ and $g$ intertwines the $\mc{W}^{wu,g}_i$ and $\mc{W}^{wu,f}_i$ foliations.
\end{lemma}
\begin{proof}
This proposition follows along the same lines as in \cite[Prop.~4.8]{dewitt2021local}, which proves the same result under the additional assumption that $g$ is $C^1$ close to $f$; this assumption is not necessary as it is only used to ensure the existence of a dominated splitting.

Pick the neighborhood $\mc{U}$ so that the conclusions of Lemma \ref{lem:dominated_splitting_exists} hold. Fix some index $1\le i<\dim E^u$. We want to show that the weak foliation $\mc{W}^{wu,g}_{i}$ exists. Let $j=\dim E^{u}-i+1$ be the corresponding index for the Lyapunov exponent of $L$ on the bundle $E^{u,L}_i$.  Suppose that $x$ and $y$ are two points in the same $\mc{W}^{u,L}_i$ leaf. Then it holds that there exists $C_{x,y}>0$ such that: 
\begin{equation}
d_{\mc{W}^{u,L}}(L^n(x), L^n(y))\le C_{x,y}e^{n\lambda_j}.
\end{equation}
Further, if $z$ and $w$ are two points in the same $\mc{W}^{u,L}$ leaf that do not lie in the same $\mc{W}^{wu,L}_i$ leaf, then due to the linearity of the dynamics of $L$ it is easy to see that these points separate at a rate $\lambda_{j-1}$ that is strictly larger than $\lambda_j$. Namely, there exists $C_{z,w}$ such that:
\begin{equation}
d_{\mc{W}^{u,L}}(L^n(z),L^n(w))\ge C_{z,w}e^{n\lambda_{j-1}}.
\end{equation}
In particular, this implies that if $x$ and $y$ are two points in the same $\mc{W}^{u,L}$ leaf such that $L^n(x)$ and $L^n(y)$ separate at rate at most $e^{n\eta}$ for some $\eta<\lambda_{j-1}$, then $x\in \mc{W}^{wu,L}_i(y)$.

We now show how one may integrate the bundle $E^{wu,g}_i$ by using these facts. Suppose that $x$ and $y$ 
 are two points in the same $\mc{W}^{u,g}$ leaf that are connected by a curve $\gamma$ that is tangent to $E^{wu,g}_i$, then by equation \eqref{eqn:narrow_bnd_spectrum} we see that as we apply $g$ that the length of $g^n\gamma$ and hence the distance between $x$ and $y$ is growing at most at rate $e^{n(\lambda_j+\epsilon)}$, i.e. there exists a constant $C_{x,y}'$ such that
\begin{equation}\label{eqn:rate_of_separation}
d_{\mc{W}^{wu,g}_i}(g^n(x),g^n(y))\le C_{x,y}'e^{n(\lambda_{j}+\epsilon)}.
\end{equation}
Let $h$ be the conjugacy between $g$ and $L$. From \cite[Cor.~4.7]{dewitt2021local}, the conjugacy $h$ induces a quasi-isometry between leaves of the $\mc{W}^{u,g}$ foliation and leaves of the $\mc{W}^{u,L}$ foliation. What this means is that there exists constants $A\ge 1$ and $B>0$ such that if $z,w$ are two points in the same $\mc{W}^{u,g}$ leaf then 
\begin{equation}\label{eqn:QI_h}
A^{-1}d_{\mc{W}^{u,g}}(z,w)-B\le d_{\mc{W}^{u,L}}(h(z),h(w))\le Ad_{\mc{W}^{u,g}}(z,w)+B.
\end{equation}

Because $h\circ g^n=L^n\circ h$, by combining equations \eqref{eqn:rate_of_separation} and \eqref{eqn:QI_h}, we deduce that there exists some $D_{z,w}>0$ such that 
\begin{equation}
d_{\mc{W}^{u,L}}(L^nh(x),L^nh(y))\le D_{z,w}e^{n(\lambda_{j}+\epsilon)},
\end{equation}
where $\lambda_j+\epsilon<\lambda_{j-1}$ by choice of the neighborhood $\mc{U}$. Then by the last sentence of the first paragraph of the proof, we conclude that $h(x)\in \mc{W}^{wu,g}_i$. This means that any curve $\gamma$ tangent to the $E^{wu,g}_i$ distribution is carried by $h$ to a curve lying in $\mc{W}^{wu,g}_i$. From this it is straightforward to conclude that $E^{wu,g}_i$ integrates to a foliation $\mc{W}^{wu,g}_i$ and that $h$ carries this foliation to the foliation $\mc{W}^{wu,L}_i$. 

We have now shown that the needed conclusions hold for the conjugacy between $g$ and $L$. By composing the conjugacy between $f$ and $L$ with the conjugacy between $g$ and $L$, we finish.
\end{proof}

We are now ready for the proof of Theorem \ref{thm:non_linear_rigidity}. The main technical point in concluding is ensuring that strong unstable foliations are carried to strong unstable foliations by the conjugacy.

\begin{proof}[Proof of Theorem \ref{thm:non_linear_rigidity}.]

At this point from Lemmas \ref{lem:dominated_splitting_exists} and \ref{lem:weak_foliations_exist}, we see that there is a $C^1$ neighborhood of $L$, $\mc{U}$, in $\Diff^{2}(\mathbb{T}^d)$ such that all of the structures and estimates described in those lemmas hold. Let $f\in \mc{U}$ and $g$ be a $C^2$ Anosov diffeomorphism with the same periodic data as $f$. Then we must show that $f$ and $g$ are $C^{1+\text{H\"older}}$ conjugate. We already have from Lemma \ref{lem:weak_foliations_exist} that the weak flags of both $f$ and $g$ exist and are intertwined by the conjugacy.

We now use the periodic data assumption to show that the strong unstable foliations are also intertwined under the conjugacy. This step is an induction argument and carries over from~\cite{gogolev2008smooth} without any changes. It relies on Property $\mathcal A$ and the orientability of $f$-invariant foliations. 
This argument at its core studies the holonomies of a strong unstable foliation between leaves of a weak unstable foliation. One deduces that these holonomies are necessarily isometric with respect to the affine parameters on the $\mc{W}^{u,f}_i$ leaves. From this, one obtains a contradiction if $h$ does not carry the strong foliation of $g$ to the strong foliation of $f$.
In the process of induction one also obtains that the conjugacy is $C^{1+\textup{H\"older}}$ along all one-dimensional expanding foliations $\mc{W}_i^{u,f}$.

The last step concludes global regularity of the conjugacy and its inverse. This step is standard and is an inductive application of the Journ\'{e}'s Lemma~\cite{journe1988regularity}. First, Journ\'{e}'s Lemma needs to be applied inductively along the weak flag to obtain $C^{1+\textup{H\"older}}$ smoothness of the conjugacy and its inverse along the unstable foliation. Then the whole argument has to be repeated to obtain $C^{1+\textup{H\"older}}$ smoothness along the stable foliation. Finally, one more application of Journ\'{e}'s Lemma finishes the proof.
\end{proof}

The proof of Theorem \ref{thm:nonlinear_rigidity_3} follows along similar lines after a preparatory step dealing with the periodic data.

\begin{proof}[Proof of Thm. \ref{thm:nonlinear_rigidity_3}]
This follows from  \cite[Thm.~2]{gogolev2008differentiable}, which says that if two Anosov diffeomorphisms of $\mathbb{T}^3$ have the same periodic data, and each have a hyperbolic splitting of their $2$-dimensional unstable bundle, then the two are $C^{1+\text{H\"older}}$ conjugate. Thus the theorem will follow from Lemma \ref{lem:dominated_splitting_exists} and its sublemma \ref{lem:holder_control}, which gives uniform H\"older control on the stable and unstable bundles from just the periodic data. To conclude, one just needs to choose $\delta$ small enough so that the splitting of such a diffeomorphism is uniformly H\"older, and then take $\delta$ even smaller so that having $\delta$-narrow spectrum around $\lambda_1>\lambda_2>0>\lambda_3$ implies that there is a splitting of $E^u$ as required by \cite[Thm.~2]{gogolev2008differentiable}.
\end{proof}

\subsection{Rigidity of Anosov automorphisms}

In the case of Anosov automorphisms, we can weaken the restriction on the dimension of the Lyapunov subspaces of the automorphisms. In addition, we can also obtain higher regularity of the conjugacy.

We begin with showing global $C^{\infty}$ rigidity of Anosov automorphisms of $\mathbb{T}^3$. We now record the following observation.

\begin{observation}\label{obs:c1_conj}
    If $f$ is $C^1$ conjugate to $L$ by a $C^1$ conjugacy $h$, then it is $C^{\infty}$ conjugate to a map $\wt{f}$ that is $C^1$ close to $L$.  To obtain this, conjugate $f$ by $\wt{h}$, a $C^{\infty}$ diffeomorphism that is $C^1$ close to $h$, so that $f$ is $C^{\infty}$ conjugated to $\wt{h}f\wt{h}^{-1}$, which is  a $C^{\infty}$ Anosov diffeomorphism that is $C^1$ close to $L$. 
\end{observation}

We can now use the above proposition to conclude our main theorem.

\begin{proof}[Proof of Theorem \ref{thm:rigidity_dim_3}]

 It suffices to consider the case that $L$ has a two dimensional unstable bundle. Then either $L$ is conformal on its unstable bundle or its unstable bundle has hyperbolic splitting. We analyze each case separately. Let $f$ denote an Anosov diffeomorphism with the same periodic data as $L$ and $h$ a conjugacy between $f$ and $L$.

There are then two cases:

\begin{enumerate}
    \item 
 If the unstable bundle is conformal then we may conclude by work of Kalinin and Sadovskaya. Let $f$ denote the other Anosov diffeomorphism with the same periodic data as $L$. Then the restriction of $Df$ to the unstable bundle of $f$ preserves a conformal structure due to \cite[Thm.~1.3]{kalinin2010linear}. Then as a conjugacy between $f$ and $L$ intertwines the stable and unstable foliations of these maps. By applying the argument in \cite[Prop.~2.3]{gogolev2011local}, this implies that $h$ is $C^{1+\text{H\"older}}$ along the stable and unstable foliations of $f$ and hence $h$ is $C^{1+\text{H\"older}}$ by Journ\'{e}'s lemma \cite{journe1988regularity}. 
 The result \cite[Cor.~2.5]{kalinin2009anosov} implies immediately that a $C^1$ conjugacy between $L$ and a $C^1$ close Anosov diffeomorphism is $C^{\infty}$. Thus we may now conclude by Observation \ref{obs:c1_conj}.
 \item 
 If the unstable bundle isn't conformal, then Theorem \ref{thm:nonlinear_rigidity_3} implies that $h$ is $C^{1+\text{H\"older}}$.  That the conjugacy is then $C^{\infty}$ now follows from the main result of \cite{gogolev2017bootstrap}, which says that a $C^{\infty}$ Anosov diffeomorphism that is $C^1$ close to $L$ and is $C^1$ conjugate to $L$ is $C^{\infty}$ conjugate to $L$. By applying Observation \ref{obs:c1_conj} to the result of the second author, we are done.
\end{enumerate}
Having completed the analysis of the two cases, we may now conclude.
 \end{proof}

We now turn to the proof of Proposition \ref{cor:generic_anosov_automorphism}. We will not give a full proof as it follows the same lines as the proof of Theorem \ref{thm:non_linear_rigidity}. 

\begin{proof}[Sketch of proof of Prop.~\ref{cor:generic_anosov_automorphism}.]

The main difference from the setting of Theorem \ref{thm:non_linear_rigidity}, is that in this setting, the intermediate foliations $\mc{W}^{u,L}_i$ may be two dimensional. That these foliations exist and have the desired properties follows from the same development in \cite{dewitt2021local} that we have mentioned before, which is indifferent to the dimensions of the leaves of the foliation. However, the argument for showing that strong foliations intertwine with strong foliations is slightly different, but follows the same lines as the argument in \cite[Prop.~2.3]{gogolev2011local}, which also deals with two dimensional foliations. The part of the argument that shows that the strong foliation is carried to the strong foliation follows by studying the holonomies of the strong foliation $\mc{W}^{uu,f}_{i+1}$ between leaves of the $\mc{W}^{u,f}_i$ foliation. By using that the differential of $Df\vert_{E^{u,f}_i}$ and $DL\vert_{E^{u,L}_i}$ can both be shown to be conformal, one may deduce that the holonomy is isometric. This argument does not use any global properties of $f$ and hence may be applied in this case as well.
\end{proof}

We do not sketch the proof of Proposition \ref{cor:nilmanifolds} as the proof is \emph{mutatis mutandis} the same as the one in \cite{dewitt2021local} and is quite similar to the two we have already described.

\section{Remaining Questions}

In view of the aforementioned results. Let us summarize some of the remaining questions in the study of periodic data rigidity of hyperbolic toral automorphisms. As these questions are central in the rigidity of hyperbolic systems and are the motivation for much research, we hope others in the field, and particularly newcomers, will find these precise statements useful. The first question is perhaps the most fundamental.
\begin{quest}
Suppose that $L\in \GL(d,\Z)$ is an irreducible Anosov automorphism. If $f$ is a $C^1$ small $C^2$ perturbation of $L$ with the same periodic data as $L$ are $f$ and $L$ $C^{1+\text{H\"older}}$ conjugate?
\end{quest}
As mentioned above, our results show that a positive answer to the above question should provide a positive answer to the global question as well. As was mentioned in the introduction, Zhenqi Wang has recently announced that the answer to this question is ``Yes" if the perturbation is $C^{\infty}$ small.

The next main question asks whether conjugacies are typically smooth. The following is referred to as the ``bootstrapping problem."

\begin{quest}
Suppose that $L$ is an irreducible Anosov automorphism and that $f$ is a $C^{\infty}$ Anosov diffeomorphism that is $C^1$ conjugate to $L$ by a conjugacy $h$. Then is $h$ $C^\infty$?
\end{quest}
The result of \cite{kalinin2023smooth} shows that the answer is ``Yes" when $f$ is sufficiently $C^\infty$ close to $L$. 

In view of the disparity of regularity between what seems possible and what is known, it is interesting to ask whether all these results work for merely $C^1$ small perturbations or even globally. In fact, our Theorem \ref{thm:rigidity_dim_3} shows that in dimension 3 that perturbative restrictions are not necessary. However, it is entirely possible that in higher dimensions the situation is much more subtle. For example, the rigidity results of \cite{kalinin2023smooth} apply even to Anosov automorphisms that are \emph{not} periodic data rigid, such as those containing Jordan blocks.

\bibliographystyle{amsalpha}
\bibliography{biblio.bib}

\end{document}